\documentclass[final]{siamltex}

\title{ Remarks on an operator Wielandt inequality}

\author{Pingping Zhang\thanks{Faculty of Science,
Chongqing University of Posts and
Telecommunications, Chongqing, 400065, People's Republic of China ({\tt E-mail: zhpp04010248@126.com}).}}

\begin{document}

\maketitle

\begin{abstract}
 Let $A$ be a positive operator on a Hilbert
space $\mathcal{H}$ with $0<m\leq A\leq M$ and $X$ and $Y$ are two
 isometries on $\mathcal{H}$ such that $X^{*}Y=0$. For every 2-positive linear map
$\Phi$, define
$$\Gamma=\left(\Phi(X^{*}AY)\Phi(Y^{*}AY)^{-1}\Phi(Y^{*}AX)\right)^{p}\Phi(X^{*}AX)^{-p}, \ \ \ p>0.$$
We consider several upper bounds for
$\frac{1}{2}|\Gamma+\Gamma^{*}|$. These bounds complement a recent
result on operator Wielandt inequality.
\end{abstract}



\pagestyle{myheadings}
\thispagestyle{plain}
\markboth{PINGPING ZHANG}{Remarks on an operator Wielandt inequality}

\section{Introduction}
As customary, we reserve $M$, $m$ for scalars and $I$ for the
identity operator. Other capital letters denote general elements of
the $C^{*}$-algebra $\mathcal{B}(\mathcal{H})$ (with unit) of all
bounded linear operator acting on a Hilbert space $(\mathcal{H},
\langle\cdot, \cdot\rangle)$. Also, we identify a scalar with the
unit multiplied by this scalar. The operator norm is denoted by
$\|\cdot\|$. In this article, the inequality between operators is in
the sense of Loewner partial order, that is, $T\geq S$ (the same as
$S\leq T$) means that $T-S$ is positive. A positive invertible
operator $T$ is naturally denoted by $T>0$.

A linear map $\Phi$ is positive if $\Phi(A)\geq0$ whenever $A\geq0$.
 We say that $\Phi$ is
2-positive if whenever the $2\times 2$ operator matrix $\left[
  \begin{array}{cc}
 A&B\\
 B^{*}&C\\
  \end{array}
\right]$ is positive, then so is $\left[
  \begin{array}{cc}
 \Phi(A)&\Phi(B)\\
 \Phi(B^{*})&\Phi(C)\\
  \end{array}
\right]$.

The Wielandt inequality \cite[p.443]{HJ1} states that if $0<m\leq
A\leq M$, and $x, y\in\mathcal{H}$ with $x\perp y$, then
$$|\langle x, Ay\rangle|^{2}\leq (\frac{M-m}{M+m})^{2}\langle x, Ay\rangle \langle y, Ax\rangle.$$

Bhatia and Davis \cite[Theorem 2]{BD} proved that if $0<m\leq A\leq
M$ and $X$ and $Y$ are two partial isometries on $\mathcal{H}$ whose
final spaces are orthogonal to each other, then for every 2-positive
linear map $\Phi$
\begin{equation}\label{equ:1}
\Phi(X^{*}AY)\Phi(Y^{*}AY)^{-1}\Phi(Y^{*}AX)\leq(\frac{M-m}{M+m})^{2}\Phi(X^{*}AX).
\end{equation}

The inequality (\ref{equ:1}) is an operator version of Wielandt inequality. In
2013, Lin presented the following conjecture.
\\\textbf{Conjecture}\cite[Conjecture 3.4]{L1}\label{con} Let $0<m\leq A\leq M$ and $X$ and $Y$
are two partial isometries on $\mathcal{H}$ whose final spaces are
orthogonal to each other. Then for every 2-positive linear map
$\Phi$
\begin{eqnarray*}
\|\Phi(X^{*}AY)\Phi(Y^{*}AY)^{-1}\Phi(Y^{*}AX)\Phi(X^{*}AX)^{-1}\|\leq
(\frac{M-m}{M+m})^{2}.
\end{eqnarray*}

Under the assumption that the conjecture is valid, the following two
inequalities follow without much additional effort
\begin{equation}\label{equ:2}
\frac{1}{2}|\Gamma+\Gamma^{*}|\leq(\frac{M-m}{M+m})^{2},
\end{equation}
and
\begin{equation}\label{equ:3}
\frac{1}{2}(\Gamma+\Gamma^{*})\leq(\frac{M-m}{M+m})^{2},
\end{equation}
where
$\Gamma=\Phi(X^{*}AY)\Phi(Y^{*}AY)^{-1}\Phi(Y^{*}AX)\Phi(X^{*}AX)^{-1}.$

If one can drop the strong assumption on the validity of the
conjecture which is still open, then apparently (\ref{equ:2}) and
(\ref{equ:3}) are very attractive alternative operator versions of
Wielandt inequality. Thus, it is interesting to estimate
$\frac{1}{2}|\Gamma+\Gamma^{*}|$ and
$\frac{1}{2}(\Gamma+\Gamma^{*})$ without the assumption that the
conjecture is true. This is our main motivation.

In this article, we consider a more general $\Gamma$ defined as
$$\Gamma=\left(\Phi(X^{*}AY)\Phi(Y^{*}AY)^{-1}\Phi(Y^{*}AX)\right)^{p}\Phi(X^{*}AX)^{-p}, \ \ \ \ p>0,$$
where $X, Y$ are isometries such that $X^{*}Y=0$.
Three bounds for both $\frac{1}{2}|\Gamma+\Gamma^{*}|$ and
$\frac{1}{2}(\Gamma+\Gamma^{*})$ are presented.

\section{Main results}

\label{} The following two lemmas are needed in our main results.
\begin{lemma}\cite[Lemma 3.5.12]{HJ2}\label{lem:3}
 For any bounded operator $X$,
 \begin{eqnarray*}
   |X|\leq tI \ \ \ \Leftrightarrow \ \ \  \|X\|\leq t \ \ \ \Leftrightarrow \ \ \ \left[
  \begin{array}{cc}
 tI&X\\
 X^{*}&tI\\
  \end{array}
\right]\geq0.
 \end{eqnarray*}
 \end{lemma}

\begin{lemma}\cite[Theorem 6]{FINS}\label{lem:4}
 Let $0\leq A\leq B$ and $0<m \leq B\leq M$. Then
 \begin{eqnarray*}
A^{2}\leq\frac{(M+m)^{2}}{4Mm}B^{2}.
\end{eqnarray*}
 \end{lemma}

Now we present the main results of this paper.
\begin{theorem}\label{theo:1}
 Let $0<m\leq A\leq M$ and let $X$ and $Y$
be two  isometries such that $X^{*}Y=0$. For every 2-positive linear map
$\Phi$, define
$$\Gamma=\left(\Phi(X^{*}AY)\Phi(Y^{*}AY)^{-1}\Phi(Y^{*}AX)\right)^{p}\Phi(X^{*}AX)^{-p}, \,\,\, p>0.$$ Then
\begin{eqnarray} \label{equ:4} \frac{1}{2}|\Gamma+\Gamma^{*}|
\leq \frac{(\frac{M-m}{M+m})^{4p}M^{2p}+m^{-2p}}{2},
\end{eqnarray}
and
\begin{eqnarray} \label{equ:5} \frac{1}{2}(\Gamma+\Gamma^{*})
\leq \frac{(\frac{M-m}{M+m})^{4p}M^{2p}+m^{-2p}}{2}.
\end{eqnarray}
\end{theorem}
\begin{proof}
 We need the following simple fact
\begin{eqnarray}\label{fact}
\|AB+BA\|\leq \|A^{2}+B^{2}\|, \ \ \  if \ A\geq0 \  \textrm{and} \
B\geq0.
\end{eqnarray}
Thus with the inequality (\ref{fact}), it follows from the
inequality (\ref{equ:1}) that
\begin{eqnarray}\label{equ:6}\frac{1}{2}\|\Gamma+\Gamma^{*}\|
&\leq&
\frac{1}{2}\left\|\left(\Phi(X^{*}AY)\Phi(Y^{*}AY)^{-1}\Phi(Y^{*}AX)\right)^{2p}+\Phi(X^{*}AX)^{-2p}\right\|\nonumber\\
&\leq&
\frac{1}{2}\left(\|\Phi(X^{*}AY)\Phi(Y^{*}AY)^{-1}\Phi(Y^{*}AX)\|^{2p}+\|\Phi(X^{*}AX)^{-1}\|^{2p}\right)\nonumber\\
&\leq&\frac{1}{2}\left(\|(\frac{M-m}{M+m})^{2}\Phi(X^{*}AX)\|^{2p}+\|\Phi(X^{*}AX)^{-1}\|^{2p}\right)\nonumber\\
 &\leq& \frac{(\frac{M-m}{M+m})^{4p}M^{2p}+m^{-2p}}{2}.
\end{eqnarray}
The last inequality above holds since $X$ is partial isometric and
$0<m\leq A\leq M$, then $m\leq\Phi(X^{*}AX)\leq M, \frac{1}{M}\leq
\Phi(X^{*}AX)^{-1}\leq\frac{1}{m}$.  Combining Lemma \ref{lem:3} and
the inequality (\ref{equ:6}) leads to the inequality (\ref{equ:4}).
The inequality (\ref{equ:5}) holds from (\ref{equ:4}) and the fact
that $|X|\geq X$ for any self-adjoint $X$. Thus, we complete the
proof.
 \end{proof}

Sharper result is presented as follows.
\begin{theorem}\label{theo:2}
 Under the same condition as in Theorem \ref{theo:1}. Then
\begin{eqnarray}\label{equ:7}\frac{1}{2}|\Gamma+\Gamma^{*}|
\leq\left\{
\begin{array}{l}(\frac{M-m}{M+m})^{2p}, \ \ \ \ \ \ \ \ \ \ \ \ \ \ \ \ when \ \ 0< p\leq \frac{1}{2},\\
\\
(\frac{M-m}{M+m})^{2p}(\frac{M}{m})^{p},\ \ \ \ \ \ \ \ \ \ \ \ \ when  \ \ p>\frac{1}{2}.\\
\end{array}
\right.
\end{eqnarray}
Consequently,
\begin{eqnarray}\label{equ:8}\frac{1}{2}(\Gamma+\Gamma^{*})
\leq\left\{
\begin{array}{l}(\frac{M-m}{M+m})^{2p}, \ \ \ \ \ \ \ \ \ \ \ \ \ \ \ \ when \ \ 0< p\leq \frac{1}{2},\\
\\
(\frac{M-m}{M+m})^{2p}(\frac{M}{m})^{p},\ \ \ \ \ \ \ \ \ \ \ \ \ when  \ \ p>\frac{1}{2}.\\
\end{array}
\right.
\end{eqnarray}
\end{theorem}
\begin{proof}
 Combining the inequality (\ref{equ:1}) and the fact that
$f(t)=t^{r}$ is operator monotone for $0\leq r\leq 1$, we have
\begin{eqnarray}\label{equ:9}
\left(\Phi(X^{*}AY)\Phi(Y^{*}AY)^{-1}\Phi(Y^{*}AX)\right)^{p}\leq
(\frac{M-m}{M+m})^{2p}\Phi(X^{*}AX)^{p},
 \end{eqnarray}
when $0< p \leq1$.\\
The inequality (\ref{equ:9}) is equivalent to
\begin{eqnarray} \label{equ:10}
\|\Gamma\|\leq (\frac{M-m}{M+m})^{2p},
 \end{eqnarray}
with $0< p \leq \frac{1}{2}$.
\\By Lemma \ref{lem:3} and the inequality (\ref{equ:10}), we have
\begin{eqnarray*}
 \left[
  \begin{array}{cc}
 (\frac{M-m}{M+m})^{2p}I&\Gamma\\
 \Gamma^{*}&(\frac{M-m}{M+m})^{2p}I\\
  \end{array}
\right]\geq0 \ \ \ \ \ \textrm{and} \ \ \ \ \left[
  \begin{array}{cc}
 (\frac{M-m}{M+m})^{2p}I&\Gamma^{*}\\
 \Gamma&(\frac{M-m}{M+m})^{2p}I\\
  \end{array}
\right]\geq0.
 \end{eqnarray*}
Summing up these two operator matrices, we have
\begin{eqnarray*}
 \left[
  \begin{array}{cc}
 2(\frac{M-m}{M+m})^{2p}I&\Gamma+\Gamma^{*}\\
 \Gamma^{*}+\Gamma&2(\frac{M-m}{M+m})^{2p}I\\
  \end{array}
\right]\geq0.
 \end{eqnarray*}
It follows from Lemma \ref{lem:3} again that
\begin{eqnarray}\label{equ:11}
 \frac{1}{2}|\Gamma+\Gamma^{*}|\leq(\frac{M-m}{M+m})^{2p},
 \end{eqnarray}
when $0<p\leq \frac{1}{2}$.\\
When $p> \frac{1}{2}$, it follows that
\begin{eqnarray}\label{equ:12}
\|\Gamma\|&\leq&
\|\left(\Phi(X^{*}AY)\Phi(Y^{*}AY)^{-1}\Phi(Y^{*}AX)\right)^{p}\|\|\Phi(X^{*}AX)^{-p}\|\nonumber\\
&=&\|\Phi(X^{*}AY)\Phi(Y^{*}AY)^{-1}\Phi(Y^{*}AX)\|^{p}\|\Phi(X^{*}AX)^{-1}\|^{p}\nonumber\\
&\leq&\|(\frac{M-m}{M+m})^{2}\Phi(X^{*}AX)\|^{p}\|\Phi(X^{*}AX)^{-1}\|^{p}  \quad \hbox{By the inequality (\ref{equ:1})}\nonumber\\
&\leq&(\frac{M-m}{M+m})^{2p}(\frac{M}{m})^{p}.
 \end{eqnarray}
Similarly, we can obtain
\begin{eqnarray}\label{equ:13}
\frac{1}{2}|\Gamma+\Gamma^{*}|
\leq(\frac{M-m}{M+m})^{2p}(\frac{M}{m})^{p},
 \end{eqnarray}
when $p>\frac{1}{2}$. \\
The inequality (\ref{equ:7}) holds by the inequalities
(\ref{equ:11}) and (\ref{equ:13}). While the inequality
(\ref{equ:8}) holds by the inequality (\ref{equ:7}). This completes
the proof of Theorem \ref{theo:2}.
\end{proof}

{\em Remark $1$.}  We provide a second method to prove the inequality
(\ref{equ:7}):
$$\|\frac{\Gamma+\Gamma^{*}}{2}\|\leq \|\Gamma\|,$$
which together with (\ref{equ:10}), (\ref{equ:12}) and Lemma
\ref{lem:3} leads to the inequality (\ref{equ:7}).

{\em Remark $2$.}\label{rem:2} The bounds in Theorem \ref{theo:2} are tighter than
those in Theorem \ref{theo:1} by the following inequalities
\begin{eqnarray*}
\frac{(\frac{M-m}{M+m})^{4p}M^{2p}+m^{-2p}}{2}\geq(\frac{M-m}{M+m})^{2p}(\frac{M}{m})^{p}\geq(\frac{M-m}{M+m})^{p}.
 \end{eqnarray*}
The first inequality above holds by the scalar AM-GM inequality, and the
 second inequality holds since $m\leq M$ and $p>0$.

Our last result is the following.
\begin{theorem}\label{theo:3}
 Under the same condition as in Theorem \ref{theo:1}. Then
\begin{eqnarray}\label{equ:15}\frac{1}{2}|\Gamma+\Gamma^{*}|
\leq(\frac{M-m}{M+m})^{2p}\left(\frac{(\frac{M}{m})^{\frac{p}{2}}+(\frac{m}{M})^{\frac{p}{2}}}{2}\right)^{\lceil
p\rceil}, \quad p>0,
\end{eqnarray}
where $\lceil p \rceil$ is the ceiling function of $p$.\\
 Consequently,
\begin{eqnarray}\label{equ:16}\frac{1}{2}(\Gamma+\Gamma^{*})
\leq(\frac{M-m}{M+m})^{2p}\left(\frac{(\frac{M}{m})^{\frac{p}{2}}+(\frac{m}{M})^{\frac{p}{2}}}{2}\right)^{
\lceil p\rceil}.
\end{eqnarray}
\end{theorem}
\begin{proof} Firstly, we prove the following inequality
\begin{eqnarray}\label{equ:17}
\|\Gamma\|\leq
(\frac{M-m}{M+m})^{2p}\left(\frac{(\frac{M}{m})^{\frac{p}{2}}+(\frac{m}{M})^{\frac{p}{2}}}{2}\right)^{\lceil
p\rceil}.
 \end{eqnarray}
The proof of (\ref{equ:17}) is proved by induction on $\lceil
p\rceil$. For $\lceil p\rceil=1$, i.e., $0<p\leq 1$, combining
(\ref{equ:9}) and Lemma \ref{lem:4} gives

$\left(\Phi(X^{*}AY)\Phi(Y^{*}AY)^{-1}\Phi(Y^{*}AX)\right)^{2p}$
\begin{eqnarray*}
&\leq&(\frac{M-m}{M+m})^{4p}
\frac{(M^{p}+m^{p})^{2}}{4M^{p}m^{p}}\Phi(X^{*}AX)^{2p}\\
&=&(\frac{M-m}{M+m})^{4p}
\left(\frac{(\frac{M}{m})^{\frac{p}{2}}+(\frac{m}{M})^{\frac{p}{2}}}{2}\right)^{2}\Phi(X^{*}AX)^{2p},\\
\end{eqnarray*}
which is equivalent to
\begin{eqnarray*}
\|\Gamma\|\leq
(\frac{M-m}{M+m})^{2p}\frac{(\frac{M}{m})^{\frac{p}{2}}+(\frac{m}{M})^{\frac{p}{2}}}{2}
=(\frac{M-m}{M+m})^{2p}\left(\frac{(\frac{M}{m})^{\frac{p}{2}}+(\frac{m}{M})^{\frac{p}{2}}}{2}\right)^{\lceil
p\rceil}.
 \end{eqnarray*}
 Thus the inequality (\ref{equ:17}) holds for
$\lceil p\rceil=1$.\\

Now, suppose that the inequality (\ref{equ:17}) holds for $\lceil
p\rceil\leq k$ $(k> 1)$. That is
\begin{eqnarray*}
\|\Gamma\|\leq
(\frac{M-m}{M+m})^{2p}\left(\frac{(\frac{M}{m})^{\frac{p}{2}}+(\frac{m}{M})^{\frac{p}{2}}}{2}\right)^{
\lceil p\rceil},
 \end{eqnarray*}
which is equivalent to

$\left(\Phi(X^{*}AY)\Phi(Y^{*}AY)^{-1}\Phi(Y^{*}AX)\right)^{2p}$
\begin{eqnarray}\label{equ:20}
&\leq&
(\frac{M-m}{M+m})^{4p}\left(\frac{(\frac{M}{m})^{\frac{p}{2}}+(\frac{m}{M})^{\frac{p}{2}}}{2}\right)^{2\lceil
p\rceil} \Phi(X^{*}AX)^{2p}.
\end{eqnarray}
Combining the inequality (\ref{equ:20}) and Lemma \ref{lem:4} leads
to

$\left(\Phi(X^{*}AY)\Phi(Y^{*}AY)^{-1}\Phi(Y^{*}AX)\right)^{4p}$
\begin{eqnarray}\label{equ:21}
&\leq&
(\frac{M-m}{M+m})^{8p}\left(\frac{(\frac{M}{m})^{\frac{p}{2}}+(\frac{m}{M})^{\frac{p}{2}}}{2}\right)^{4\lceil p\rceil}\frac{(M^{2p}+m^{2p})^{2}}{4M^{2p}m^{2p}}\Phi(X^{*}AX)^{4p}\nonumber\\
&=& (\frac{M-m}{M+m})^{8p}\left(\frac{(\frac{M}{m})^{\frac{p}{2}}+(\frac{m}{M})^{\frac{p}{2}}}{2}\right)^{4 \lceil p\rceil}\left(\frac{(\frac{M}{m})^{p}+(\frac{m}{M})^{p}}{2}\right)^{2}\Phi(X^{*}AX)^{4p}\nonumber\\
&\leq&(\frac{M-m}{M+m})^{8p}\left(\frac{(\frac{M}{m})^{p}+(\frac{m}{M})^{p}}{2}\right)^{2\lceil p\rceil}\left(\frac{(\frac{M}{m})^{p}+(\frac{m}{M})^{p}}{2}\right)^{2}\Phi(X^{*}AX)^{4p}\nonumber\\
&=&(\frac{M-m}{M+m})^{8p}\left(\frac{(\frac{M}{m})^{p}+(\frac{m}{M})^{p}}{2}\right)^{2(\lceil p\rceil+1)} \Phi(X^{*}AX)^{4p}\nonumber\\
&\leq&(\frac{M-m}{M+m})^{8p}\left(\frac{(\frac{M}{m})^{p}+(\frac{m}{M})^{p}}{2}\right)^{2(k+1)}\Phi(X^{*}AX)^{4p},
\end{eqnarray}
where $\lceil p\rceil\leq k$. The inequality (\ref{equ:21}) is equivalent to

$\left(\Phi(X^{*}AY)\Phi(Y^{*}AY)^{-1}\Phi(Y^{*}AX)\right)^{2p}$
\begin{eqnarray}\label{equ:22}
&\leq&(\frac{M-m}{M+m})^{4p}\left(\frac{(\frac{M}{m})^{\frac{p}{2}}+(\frac{m}{M})^{\frac{p}{2}}}{2}\right)^{2(k+1)}\Phi(X^{*}AX)^{2p},
\end{eqnarray}
where $\lceil p\rceil\leq 2k$. Since $k+1\leq 2k$ when $k\geq 1$,
 the inequality (\ref{equ:22}) holds for $\lceil p\rceil=k+1$,
i.e.,

$\left(\Phi(X^{*}AY)\Phi(Y^{*}AY)^{-1}\Phi(Y^{*}AX)\right)^{2p}$
\begin{eqnarray*}
&\leq&(\frac{M-m}{M+m})^{4p}\left(\frac{(\frac{M}{m})^{\frac{p}{2}}+(\frac{m}{M})^{\frac{p}{2}}}{2}\right)^{2(k+1)} \Phi(X^{*}AX)^{2p}\\
&=&
(\frac{M-m}{M+m})^{4p}\left(\frac{(\frac{M}{m})^{\frac{p}{2}}+(\frac{m}{M})^{\frac{p}{2}}}{2}\right)^{2\lceil
p\rceil}\Phi(X^{*}AX)^{2p},
\end{eqnarray*}
which is equivalent to
\begin{eqnarray*}
\|\Gamma\|\leq
(\frac{M-m}{M+m})^{2p}\left(\frac{(\frac{M}{m})^{\frac{p}{2}}+(\frac{m}{M})^{\frac{p}{2}}}{2}\right)^{\lceil
p\rceil},
 \end{eqnarray*}
when $\lceil p\rceil=k+1$. Thus the inequality (\ref{equ:17}) holds
for $\lceil p\rceil=k+1$. This completes the induction.

Combining the inequality (\ref{equ:17}), Remark $1$ and
Lemma \ref{lem:3} leads to the inequality (\ref{equ:15}). Similarly,
the inequality (\ref{equ:16}) holds by the inequality
(\ref{equ:15}).
\end{proof}

{\em Remark $3$.} Since
$\frac{(\frac{M}{m})^{\frac{p}{2}}+(\frac{m}{M})^{\frac{p}{2}}}{2}\geq
1$,  the bounds in Theorem \ref{theo:2} are sharper than those
in Theorem \ref{theo:3} when $0<p\leq \frac{1}{2}$. It follows from the fact
$\frac{m}{M}\leq 1\leq \frac{M}{m}$ that $
\frac{(\frac{M}{m})^{\frac{p}{2}}+(\frac{m}{M})^{\frac{p}{2}}}{2}\leq
(\frac{M}{m})^{\frac{p}{2}}$. Thus, the bounds
in Theorem \ref{theo:3} are tighter than those in Theorem
\ref{theo:2} when $\frac{1}{2}<p\leq 2$. When  $p$ is large enough, such as $p>2+2\log_{\frac{M}{m}}2$, the bounds in Theorem \ref{theo:2} are tighter than those
in Theorem \ref{theo:3} by the following inequalities
$$(\frac{M}{m})^{\frac{p}{2}}>2\frac{M}{m}, \quad\quad if\quad p>2+2\log_{\frac{M}{m}}2,$$
and
$$\left(\frac{(\frac{M}{m})^{\frac{p}{2}}+(\frac{m}{M})^{\frac{p}{2}}}{2}\right)^{\lceil p\rceil}>\left(\frac{(\frac{M}{m})^{\frac{p}{2}}}{2}\right)^{\lceil p \rceil}> \left(\frac{(\frac{M}{m})^{\frac{p}{2}}}{2}\right)^{p}. $$
 Therefore, neither of the
bounds in Theorem \ref{theo:2} and Theorem \ref{theo:3} are
uniformly better than the other.

\end{document}